\documentclass[preprint,11pt,3p]{elsarticle}
\usepackage{amsmath,amsthm,amsfonts,amssymb,arydshln}
\usepackage{xcolor}
\usepackage{algorithm}
\usepackage{multirow}
\usepackage{mathdots}
\usepackage{multicol,graphicx}
\usepackage{mathrsfs}
\usepackage{booktabs}
\usepackage{algorithm}              
\usepackage{algorithmic}
\usepackage{comment}

\usepackage{mathrsfs}
\usepackage{sectsty}
\definecolor{astral}{RGB}{46,116,181}
\sectionfont{\color{astral}}
\newtheorem{theorem}{Theorem}[section]
\newtheorem{lemma}[theorem]{Lemma}
\newtheorem{corollary}[theorem]{Corollary}
\newtheorem{proposition}[theorem]{Proposition}

\newtheorem{definition}[theorem]{Definition}
\newtheorem{example}[theorem]{Example}

\usepackage{hyperref}
\usepackage{subfigure}
\usepackage{scalerel}
\usepackage{stackengine,wasysym}

\usepackage{tikz}
\pgfdeclarelayer{bg} 
\pgfsetlayers{bg,main}

\usepackage{tikz,xcolor}
\definecolor{lime}{HTML}{A6CE39}
\definecolor{lightblue}{rgb}{0.0, 0.0, 0.5}
\DeclareRobustCommand{\orcidicon}{
	\begin{tikzpicture}
	\draw[lime, fill=lime] (0,0)
	circle [radius=0.16]
	node[white] {{\fontfamily{qag}\selectfont \tiny ID}};
	\draw[white, fill=white] (-0.0625,0.095)
	circle [radius=0.007];
	\end{tikzpicture}
	\hspace{-2mm}
}
\foreach \x in {A, ..., Z}{
	\expandafter\xdef\csname orcid\x\endcsname{\noexpand\href{https://orcid.org/\csname orcidauthor\x\endcsname}{\noexpand\orcidicon}}
}

\newcommand{\Ind}{{\mathrm {Ind}}}

\begin{document}
\begin{frontmatter}
\title{
Perturbation Analysis of the QT-Drazin Inverse of Quaternion Tensors via the QT-Product
}
\author[a]{Yue Zhao}
\ead{yuezhao0303@163.com}

\author[b]{Daochang Zhang\corref{cor1}}
\ead{daochangzhang@126.com}
\cortext[cor1]{Corresponding author.}

\author[b]{Jingqian Li}
\ead{JingqianLi85@163.com}

\author[c]{Dijana Mosi\'{c}}
\ead{dijana@pmf.ni.ac.rs}

\address[a]{Department of Math \& Stat, Georgia State University, Atlanta, GA 30303, USA.}
\address[b]{College of Sciences, Northeast Electric Power University, Jilin, P.R. China.}
\address[c]{Faculty of Sciences and Mathematics, University of Ni\v s, P.O.
Box 224, 18000 Ni\v s, Serbia.}

\begin{abstract}
The motivation of this paper is to investigate the perturbation theory for the QT-Drazin inverse of quaternion tensors under the QT-product via the associated $z$-block circulant representation. A fundamental relationship between the QT-Drazin inverse of $\mathtt{bcirc}_z(\mathcal A)$ and the $z$-block circulant form of $\mathcal A^D$ is established. Moreover, the QT-index of a quaternion tensor is characterized by the indices of the diagonal blocks in the corresponding block-diagonalized matrix.  As a consequence, a representation of the QT-Drazin inverse in terms of the QT-Moore--Penrose inverse is derived, which offers a practical approach for its direct computation in MATLAB. Furthermore, a decomposition theory for the QT-Drazin inverse is developed by combining the structure of $z$-block circulant matrices with the Jordan decomposition of quaternion matrices. Numerical examples are provided to demonstrate the theoretical results and computational feasibility.
\end{abstract}

\begin{keyword}
Quaternion Tensors\sep QT-Product\sep $z$-Block Circulant Matrices\sep QT-Index\sep Perturbation  

\MSC[2010] 15A09; 15A69; 15A72; 65F35
\end{keyword}

\end{frontmatter}

\section{Introduction}
Tensors, as higher-order generalizations of vectors and matrices, are widely used in signal processing, computer vision, and graph analysis~\cite{SIAMMKB2013, SIAMSMS2013, WDPRESS2016, MFBC2023}. An order-\(m\) tensor can be viewed as a multidimensional array of the form
\[
\mathcal{A} = \left( a_{i_1 i_2 \cdots i_m} \right) \in \mathbb{R}^{n_1 \times n_2 \times \cdots \times n_m}.
\]
In response to the growing interest in multidimensional data analysis, several tensor products and associated decompositions have been proposed, including the Einstein product~\cite{AE2007}, T-product~\cite{SIAMMKB2013}, C-product~\cite{LAAKLA2015}, and M-product~\cite{LAAKLA2015}, etc.

The T-product enables the singular value decomposition (T-SVD) of third-order tensors but is limited in scope, as it does not naturally extend to higher-order data structures such as color videos, which are typically modeled as fourth-order tensors. However, a color video can alternatively be represented as a third-order quaternion tensor. To generalize the T-product to this setting, the QT-product was introduced in 2022~\cite{AMLzhang2022}, leading to the development of a corresponding SVD framework for third-order quaternion tensors. Moreover, the T-product between two complex tensors can be equivalently represented as the QT-product between two quaternion tensors.

In particular, tensor generalized inverses play a fundamental role in numerical multilinear algebra and have been applied in many areas of science and engineering~\cite{SIAMRevKB2009,CMALZ2019}. In 2017, Jin et al.~\cite{CMAJin2017} introduced the Moore–Penrose inverse of tensors based on the T-product. Subsequently, the Drazin inverse was investigated by Miao et al.~\cite{CAMmqw2021}. Extensive research has focused on generalized tensor inverses and their perturbation theory under various tensor product frameworks. 

Accordingly, this paper aims to advance the perturbation theory of the QT-Drazin inverse and investigate its applications within the framework of the QT-product.

The paper is organized as follows. Section 2 recalls several preliminaries and lemmas on quaternion tensors, the QT-product, and the $z$-block circulant representation. Section 3 examines the relationship between the QT-Drazin inverse of $\mathtt{bcirc}_{z}(\mathcal A)$ and the $z$-block circulant representation of $\mathcal A^{D}$. It is worth noting that we establish a characterization of the QT-index through the indices of the diagonal blocks in the associated block-diagonalized matrix. On this basis, we derive a representation of the QT-Drazin inverse via the QT-Moore--Penrose inverse, thereby providing a practical approach for its direct computation in MATLAB, and further develop a decomposition of the QT-Drazin inverse of the quaternion tensor $\mathcal A$. Section 4 is devoted to perturbation analysis for the QT-Drazin inverse under different conditions. Section 5 concludes the paper with a numerical example illustrating the theoretical results.

\section{Preliminaries and key lemmas}
In this section, the principal definitions and notations required for the later development are introduced.

Let $\mathbb{R}$ and $\mathbb{C}$ denote the fields of real and complex numbers, respectively. 
The set of quaternions, denoted by $\mathbb{Q}$, constitutes a four-dimensional vector space 
over $\mathbb{R}$ with the ordered basis $\{\mathbf{1}, \mathbf{i}, \mathbf{j}, \mathbf{k}\}$.

Each $a \in \mathbb{Q}$ admits the representation
\[
a = a_0 + a_1 \mathbf{i} + a_2 \mathbf{j} + a_3 \mathbf{k},
\]
where $a_0, a_1, a_2, a_3 \in \mathbb{R}$.
The imaginary units satisfy the standard multiplication rules
\[
\mathbf{i}^2 = \mathbf{j}^2 = \mathbf{k}^2 = -1, \qquad
\mathbf{ij} = \mathbf{k}, \quad
\mathbf{jk} = \mathbf{i}, \quad
\mathbf{ki} = \mathbf{j},
\]
and they are noncommutative:
$
\mathbf{ji} = -\mathbf{k}, \quad
\mathbf{kj} = -\mathbf{i}, \quad
\mathbf{ik} = -\mathbf{j}.
$ The conjugate of $a$ is defined by
$
a^* = a_0 - a_1 \mathbf{i} - a_2 \mathbf{j} - a_3 \mathbf{k},
$
and the norm of $a$ is
$
|a| = \sqrt{a^* a} = \sqrt{a_0^2 + a_1^2 + a_2^2 + a_3^2}.
$
Equivalently, any quaternion $a$ can be regarded as an element of $\mathbb{C}^2$ and can be written uniquely as
\[
a = c_1 + \mathbf{j} c_2,
\]
where $ c_1=a_0 + a_1 \mathbf{i},$ $c_2=a_2 - a_3 \mathbf{i},$ such that $c_1,~c_2\in \mathbb{C}.$

Similar to a single quaternion, any quaternion matrix 
$Q \in \mathbb{Q}^{m\times n}$ can be represented as
\[
Q = Q_0+Q_1\mathbf{i}+Q_2\mathbf{j}+Q_3 \mathbf{k}=Q_{\mathbf{d}} + \mathbf{j} Q_{\mathbf{c}},
\]
where $Q_0,~Q_1,~Q_2,~Q_3\in\mathbb R^{m\times n}$, $Q_{\mathbf{d}}=Q_0+Q_1\mathbf{i}\in \mathbb{C}^{m\times n}$ and $Q_{\mathbf{c}}=Q_2-Q_3\mathbf{i}\in \mathbb{C}^{m\times n}$ . 
The conjugate transpose of $\mathcal A$ is $Q^*=Q_0^*-Q_1^*\mathbf{i} -Q_2^* \mathbf{j} - Q_3^* \mathbf{k}$.

Following~\cite{LAAzfzl1997}, we introduce the eigenvalues of quaternion matrices, which, owing to the noncommutativity of quaternion multiplication, are classified into left and right types. In this work, only right eigenvalues are considered.
\begin{definition}
A scalar \(\lambda \in \mathbb{Q}\) is called a \emph{right} (respectively, \emph{left}) eigenvalue of a quaternion matrix 
\(A \in \mathbb{Q}^{n\times n}\) if there exists a nonzero vector \(x \in \mathbb{Q}^n\) satisfying 
\(A x = x \lambda\) (respectively, $A x = \lambda x$); in this case, $x$ is referred to as the corresponding right (respectively, left) eigenvector.
\end{definition}

In what follows, we introduce the definition of a third-order quaternion tensor 
and describe the associated $z$-block circulant matrix structure, following the discussion in~\cite{DY2026}. Before proceeding, we clarify several notational conventions. Bold lowercase letters denote vectors. A third-order tensor $\mathcal{A}$ can be regarded, respectively,  as a collection of three families of matrix slices: frontal, lateral, and horizontal. In particular, $\mathcal{A}^{(i)}$ represents the $i$-th frontal slice of $\mathcal{A}$, where $1 \le i \le n_3$.

\begin{definition}
 Let $\mathcal A\in \mathbb Q^{n_1\times n_2\times n_3}$, $\mathcal{A}_{\mathbf{d}}\in \mathbb C^{n_1\times n_2\times n_3}$, and $\mathcal{A}_{\mathbf{c}}\in \mathbb C^{n_1\times n_2\times n_3}$ satisfy that $$\mathcal A=\mathcal{A}_{\mathbf{d}} + \mathbf{j} \mathcal{A}_{\mathbf{c}}.$$ 
    Then the z-block circulant matrix associated with $\mathcal A$, denoted by $\mathtt{bcirc_z}(\mathcal A)$, is defined as:
    \[
    \mathtt{bcirc_z}(\mathcal A)=\mathtt{bcirc}(\mathcal{A}_{\mathbf{d}})+\mathbf{j}\mathtt{bcirc}(\mathcal{A}_{\mathbf{c}})(P_{n_3}\otimes I_{n_2}),
    \]
    where $\otimes$ denotes the Kronecker Product and $P_{n_3}$ is the $n_3\times n_3$ permutation matrix given by:
\begin{align}\label{PN3}
P_{n_3} =
\begin{bmatrix}
1&0&\cdots&0&0\\0&0&\cdots&0&1\\0&0&\cdots&1&0\\\vdots&\vdots&\iddots&\vdots&\vdots\\0&1&\cdots&0&0
\end{bmatrix}.
\end{align}
Moreover, the explicit structure of the $z$-block circulant matrix 
$\mathtt{bcirc_z}(\mathcal{A})\in\mathbb{Q}^{n_1n_3\times n_2n_3}$ is obtained as follows:
    \begin{align*}\label{bcirczA}
        \mathtt{bcirc_z}(\mathcal A)=
        \begin{bmatrix}
        \mathcal{A}_{\mathbf{d}}^{(1)} + \mathbf{j}\mathcal{A}_{\mathbf{c}}^{(1)} & \mathcal{A}_{\mathbf{d}}^{(n_3)} +  \mathbf{j}\mathcal{A}_{\mathbf{c}}^{(2)} & \cdots & \mathcal{A}_{\mathbf{d}}^{(2)} + \mathbf{j}\mathcal{A}_{\mathbf{c}}^{(n_3)} \\
        \mathcal{A}_{\mathbf{d}}^{(2)} + \mathbf{j}\mathcal{A}_{\mathbf{c}}^{(2)} & \mathcal{A}_{\mathbf{d}}^{(1)} + \mathbf{j}\mathcal{A}_{\mathbf{c}}^{(3)} & \cdots & \mathcal{A}_{\mathbf{d}}^{(3)} + \mathbf{j}\mathcal{A}_{\mathbf{c}}^{(1)} \\
        \mathcal{A}_{\mathbf{d}}^{(3)} + \mathbf{j}\mathcal{A}_{\mathbf{c}}^{(3)} & \mathcal{A}_{\mathbf{d}}^{(2)} + \mathbf{j}\mathcal{A}_{\mathbf{c}}^{(4)} & \cdots & \mathcal{A}_{\mathbf{d}}^{(4)} + \mathbf{j}\mathcal{A}_{\mathbf{c}}^{(4)} \\
        \vdots & \vdots & \ddots & \vdots \\
        \mathcal{A}_{\mathbf{d}}^{(n_3)} + \mathbf{j}\mathcal{A}_{\mathbf{c}}^{(n_3)} & \mathcal{A}_{\mathbf{d}}^{(n_3-1)} + \mathbf{j}\mathcal{A}_{\mathbf{c}}^{(1)} & \cdots & \mathcal{A}_{\mathbf{d}}^{(1)} + \mathbf{j}\mathcal{A}_{\mathbf{c}}^{(n_3-1)}
        \end{bmatrix},
    \end{align*} where $\mathcal{A}_{\mathbf{d}}^{(i)}+\mathbf{j}\mathcal{A}_{\mathbf{c}}^{(i)}\in\mathbb{Q}^{n_1\times n_2}$ is the $i$-th frontal slice of $\mathcal{A}$.
\end{definition}

As defined in~\cite{AMLzhang2022}, the block circulant matrix corresponding to a third-order quaternion tensor
$\mathcal{A}$, denoted by $\mathtt{bcirc}(\mathcal{A}) \in 
\mathbb{Q}^{n_1 n_3 \times n_2 n_3}$, takes the following form:
\begin{align*}
    \mathtt{bcirc}(\mathcal{A}) :=
\begin{bmatrix}
\mathcal A^{(1)} &\mathcal A^{(n_3)} & \mathcal A^{(n_3-1)} & \cdots & \mathcal A^{(2)} \\
\mathcal A^{(2)} & \mathcal A^{(1)} & \mathcal A^{(n_3)} & \cdots & \mathcal A^{(3)} \\
\vdots & \vdots & \vdots & \ddots & \vdots \\
\mathcal A^{(n_3)} & \mathcal A^{(n_3-1)} & \mathcal A^{(n_3-2)} & \cdots & \mathcal A^{(1)}
\end{bmatrix}.
\end{align*}
Furthermore, the operation $\mathtt{unfold}$ transforms a tensor 
$\mathcal{A} \in \mathbb{Q}^{n_1 \times n_2 \times n_3}$ 
into a block column vector of size $n_1 n_3 \times n_2$, 
obtained from the first block column of $\mathtt{bcirc}(\mathcal{A})$. 
The inverse transformation, denoted by $\mathtt{fold}$, 
recovers the tensor $\mathcal{A}$ from its unfolded representation. The two operations are formally defined as 
$\mathtt{unfold}(\mathcal{A}) := [\,\mathcal{A}^{(1)};\,\mathcal{A}^{(2)};\,\cdots;\,\mathcal{A}^{(n_3)}\,]$ 
and $\mathtt{fold}(\mathtt{unfold}(\mathcal{A})) := \mathcal{A}$, 
where each $\mathcal{A}^{(i)} \in \mathbb{Q}^{n_1 \times n_2}$.

It is a classical result that the normalized discrete Fourier transform (DFT) matrix 
can diagonalize any complex circulant matrix~\cite{JHUPgl1996}. 
Therefore, for any tensor $\mathcal{A} \in \mathbb{C}^{n_1 \times n_2 \times n_3}$, we obtain 
\begin{align*}
    (F_{n_3}\otimes I_{n_1})\mathtt{bcirc}(\mathcal A)(F_{n_3}^*\otimes I_{n_2})=\begin{bmatrix}
        \hat{\mathcal A}^{(1)}&&&\\&\hat{\mathcal A}^{(2)}&&\\&&\ddots&\\&&&\hat{\mathcal A}^{(n_3)}
\end{bmatrix}:=diag(\hat{\mathcal A}),
\end{align*}
where $\hat{\mathcal{A}}^{(s)}$ ($s = 1, 2, \ldots, n_3$) are the frontal slices of 
$\hat{\mathcal{A}}$, and $F_{n_3} \in \mathbb{C}^{n_3 \times n_3}$ denotes the normalized 
discrete Fourier transform (DFT) matrix whose entries are defined by 
\[
[F_{n_3}]_{ij} = n_3^{-1/2} \omega^{(i-1)(j-1)}, \qquad 1 \le i, j \le n_3,
\]
with $\omega = e^{-2\pi \mathbf{i}/n_3}$. 

Additionally, several basic definitions that will be used throughout the subsequent sections 
are introduced below. In particular, the definition of the QT-product is given as follows:
\begin{definition}\cite{AMLzhang2022}\label{defAB}
Let $\mathcal{A}_{\mathbf{d}},~\mathcal{A}_{\mathbf{c}} \in \mathbb{C}^{n_1 \times r \times n_3}$, $\mathcal{A} = \mathcal{A}_{\mathbf{d}} + \mathbf{j}\mathcal{A}_{\mathbf{c}} 
\in \mathbb{Q}^{n_1 \times r \times n_3},$ and 
$\mathcal{B} \in \mathbb{Q}^{r \times n_2 \times n_3}$. The symbol $\otimes$ denotes the Kronecker product. Then
\begin{align*}
    \mathcal A*_Q\mathcal B\stackrel{.}{=}\mathtt{fold}\Big(\big(\mathtt{bcirc}(\mathcal A_{\mathbf{d}})+\mathbf j\mathtt{bcirc}(\mathcal A_{\mathbf{c}})\cdot (P_{n_3}\otimes I_r)\big)\cdot\mathtt{unfold}(\mathcal B)\Big)\in \mathbb{Q}^{n_1 \times n_2 \times n_3},
    \end{align*} where $P_{n_3}$ is given by \eqref{PN3}.
\end{definition}

We next present the definitions of the identity quaternion tensor, the conjugate transpose of a quaternion tensor, and the f-diagonal quaternion tensor.
\begin{definition}\cite{AMLzhang2022}
The $n \times n \times l$ identity quaternion tensor $\mathcal{I}_{nnl}$ is the tensor whose first frontal slice is the identity matrix, and whose remaining frontal slices are zero matrices.
\end{definition}

\begin{definition}\cite{AMLzhang2022}
The conjugate transpose of a quaternion tensor 
$\mathcal{A} = \mathcal A_{\mathbf{d}} + \mathbf{j}\mathcal A_{\mathbf{c}} \in \mathbb{Q}^{n_1 \times n_2 \times n_3}$ 
is also denoted as $\mathcal{A}^* \in \mathbb{Q}^{n_2 \times n_1 \times n_3}$, satisfying
\[
\mathtt{unfold}(\mathcal{A}^*) = \mathtt{unfold}(A_{\mathbf{d}}^*) - (P_{n_3} \otimes I_{n_2}) \mathtt{unfold}(A_{\mathbf{c}}^*)\mathbf j.
\]
\end{definition}

\begin{definition}\cite{AMLzhang2022}
The tensor \(\mathcal{P}_{n n n_3} \in \mathbb{Q}^{n \times n \times n_3}\) is called an \emph{f-diagonal} quaternion tensor if each of its frontal slices is a diagonal quaternion matrix.
\end{definition}

\begin{definition}\cite{LiuAMC2026}
The tensor \(\mathcal{P} \in \mathbb{Q}^{n \times n \times n_3}\) is called an \emph{f-upper(lower)-triangular} quaternion tensor if each of its frontal slices is an upper(lower)-triangular quaternion matrix.
\end{definition} 

Subsequently, we propose the following definitions using the $z$-block circulant matrix.
\begin{definition}\label{QTSPECTRAL}
Let \(\mathcal{A} \in \mathbb{Q}^{n_1 \times n_2 \times n_3}\). The QT-spectral norm of \(\mathcal{A}\) is defined by  
\[
\|\mathcal{A}\|_s:= \|\mathtt{bcirc_z}(\mathcal{A})\|_2,
\]  
where \(\|\mathtt{bcirc_z}(\mathcal{A})\|_2\) denotes the \emph{spectral norm} of the $z$-block circulant quaternion matrix \(\mathtt{bcirc_z}(\mathcal{A})\).
\end{definition}

\begin{definition}\label{defrhoqt}
   The QT-spectral radius of $\mathcal A$ is defined as $$\scalebox{1.3}{$\rho$}_{QT}(\mathcal{A}):=\scalebox{1.3}{$\rho$}(\mathtt{bcirc_z}(\mathcal A)),$$ where $\scalebox{1.35}{$\rho$}_{QT}$ denotes the \emph{spectral radius} of a quaternion tensor.
\end{definition}

Based on the preceding definitions, 
we next introduce several lemmas that will be essential for proving the main results.
\begin{lemma}\cite{DY2026}\label{zdc}
Let $\mathcal{A}_{\mathbf{d}}$ and $\mathcal{A}_{\mathbf{c}}\in \mathbb C^{n_1\times r\times n_3}$ satisfy that $\mathcal A=\mathcal{A}_{\mathbf{d}} + \mathbf{j} \mathcal{A}_{\mathbf{c}}\in\mathbb{Q}^{{n_1\times r\times n_3}}.$ Denote that $\hat{\mathcal A}$ is the DFT of $\mathcal{A}$. Then, 
\begin{align*}
    (F_{n_3}\otimes I_{n_1})\mathtt{bcirc_z}(\mathcal A)(F_{n_3}^*\otimes I_{r})=diag(\hat{\mathcal A}).
\end{align*}
\end{lemma}

\begin{lemma}\cite{DY2026}\label{newrelation}
    Let $\mathcal A\in \mathbb Q^{n_1\times r\times n_3},$ $\mathcal B\in \mathbb Q^{r\times n_2\times n_3},$ and $\mathcal C\in \mathbb Q^{n_1\times n_2\times n_3}.$ Denote that $\hat{\mathcal A}$, $\hat{\mathcal B}$, $\hat{\mathcal C}$ are the DFT of $\mathcal A,$ $\mathcal B,$ $\mathcal C,$ respectively. Then
    \begin{align*}
       \mathcal A*_Q\mathcal B=\mathcal C\Longleftrightarrow \mathtt{bcric_z}(\mathcal A)\mathtt{bcric_z}(\mathcal B)=\mathtt{bcric_z}(\mathcal C).
    \end{align*}
    \end{lemma}

\begin{lemma}\cite{DY2026}\label{exchangeinverse}
    Let $\mathcal P\in\mathbb{Q}^{n\times n\times n_3}$. Then $\mathcal P$ is invertible if and only if $\mathtt{bcirc_z}(\mathcal P)$ is invertible.
    \end{lemma}

\begin{lemma}\cite{DY2026}\label{ct}
    Let $\mathcal A\in\mathbb Q^{n\times n\times n_3}$ and the conjugate transpose of $\mathcal A$ is defined by $\mathcal A^*$. Then 
    \begin{align*}
        \mathtt{bcirc_z}(\mathcal A^*)=\mathtt{bcirc_z}(\mathcal A)^*.
    \end{align*}
    \end{lemma}

\begin{lemma}\cite{DY2026}\label{unitary}
        A quaternion tensor $\mathcal A\in\mathbb{Q}^{n\times n\times n_3}$ is unitary if and only if $\mathtt{bcirc_z}(\mathcal A)$ is unitary.
\end{lemma}

Moreover, the following properties are presented.
\begin{lemma}\label{rhos}
Suppose that $\mathcal A\in\mathbb Q^{n_1\times m\times n_3}$, $\mathcal B\in\mathbb Q^{m\times n_2\times n_3}$. Then
\begin{enumerate}
\item $\scalebox{1.3}{$\rho$}_{QT}(\mathcal{A})\leq \|\mathcal A\|_s;$
\item $\|\mathcal A*_Q\mathcal B\|_s\leq \|\mathcal A\|_s\|\mathcal B\|_s;$
\item $(\mathcal A*_Q\mathcal B)^*=\mathcal B^**_Q\mathcal A^*;$
\item $(\mathcal A*_Q\mathcal B)*_Q\mathcal C=\mathcal A*_Q(\mathcal B*_Q\mathcal C), \text{where $\mathcal C\in\mathbb Q^{n_2\times n\times n_3}$};$
\item $\mathcal A*_Q(\mathcal B+\mathcal C)=\mathcal A*_Q\mathcal B+\mathcal A*_Q\mathcal C, \text{where $\mathcal C\in\mathbb Q^{m\times n_2\times n_3}$}.$
\end{enumerate}
\begin{proof}
    The first inequality is directly follows by Definition \ref{QTSPECTRAL} and Definition \ref{defrhoqt}. 
    
    In addition, $\|\mathtt{bcirc_z}(\mathcal{A})\mathtt{bcirc_z}(\mathcal{B})\|_2\leq\|\mathtt{bcirc_z}(\mathcal{A})\|_2\|\mathtt{bcirc_z}(\mathcal{B})\|_2$ follows directly from the properties of the matrix norm. By Lemma \ref{newrelation}, we have $\|\mathtt{bcirc_z}(\mathcal{A})\mathtt{bcirc_z}(\mathcal{B})\|_2=\|\mathtt{bcirc_z}(\mathcal{A}*_Q\mathcal{B})\|_2$. From Definition \ref{QTSPECTRAL}, we conclude that $\|\mathcal{A}*_Q\mathcal{B}\|_s\leq\|\mathcal A\|_s\|\mathcal B\|_s$. 
    
    Furthermore, by Lemma \ref{newrelation} and Lemma \ref{ct}, we obtain 
    \begin{align*}
     \mathtt{bcirc_z}\big((\mathcal A*_Q\mathcal B)^*\big)
     &=\mathtt{bcirc_z}(\mathcal A*_Q\mathcal B)^*=\big(\mathtt{bcirc_z}(\mathcal A)\mathtt{bcirc_z}(\mathcal B)\big)^*
     \\&=\mathtt{bcirc_z}(\mathcal B)^*\mathtt{bcirc_z}(\mathcal A)^*=\mathtt{bcirc_z}(\mathcal B^*)\mathtt{bcirc_z}(\mathcal A^*),
    \end{align*} which implies $(\mathcal A*_Q\mathcal B)^*=\mathcal B^**_Q\mathcal A^*.$ 
    
    The remaining steps follow by analogous arguments and are omitted for brevity.
\end{proof}
\end{lemma}

\section{QT-Drazin inverse of $\mathtt{bcirc_z}(\mathcal A)$}

In this section, we derive several key results concerning the QT-Drazin inverse, which are instrumental for the perturbation analysis that follows. We begin by introducing the QT-rank, QT-index, and QT-Drazin inverse of a third-order quaternion tensor $\mathcal A$.

\begin{definition}\cite{AMLzhang2022}\label{rankqt}
     Let $\mathcal A\in\mathbb{Q}^{n_1\times n_2\times n_3}$ and its QT-SVD $\mathcal A=\mathcal U*_Q\mathcal S *_Q\mathcal V^*$. The QT-rank of $\mathcal A$ is the number of nonzero elements of ${\mathcal S(i,i,:)}_{i=1}^{m}$, where $m\stackrel{\cdot}{=}min(n_1,n_2)$. Specifically,
    \begin{align*}
    \operatorname{rank}_{QT}(\mathcal A)\stackrel{\cdot}{=}\#\{i|\|\mathcal S(i,i,:)\|_F>0\}.
    \end{align*}
     Furthermore, the $i$-th singular value of $\mathcal A$ is defined by $\sigma_i(\mathcal A)\stackrel{\cdot}{=}\|S(i,i,:)\|_F$ for $i\in[m]$.
\end{definition}

With the preceding definition in place, we next define the QT-index.
\begin{definition}\label{QT_index}\cite{PJOOma2024}
    Let $\mathcal{A}\in\mathbb{Q}^{n\times n\times n_3}$. The \emph{QT-index} of $\mathcal{A},$ written as $\operatorname{Ind}_{QT}(\mathcal{A}),$ is defined to be the minimal nonnegative integer $k$ for which
\[
\operatorname{rank}_{QT}(\mathcal{A}^{k+1})=\operatorname{rank}_{QT}(\mathcal{A}^k).
\]
\end{definition}
Then, we state the definition of the QT-Drazin inverse.
\begin{definition}\cite{PJOOma2024}
Let $\mathcal{A}\in\mathbb{Q}^{n\times n\times n_3}$ with $\operatorname{Ind}_{QT}(\mathcal{A})=k$. A quaternion tensor $\mathcal{X}\in\mathbb{Q}^{n\times n\times n_3}$ is called the \emph{QT-Drazin inverse} of $\mathcal{A}$ if it satisfies
\[
\mathcal{A}^k *_Q \mathcal{X} *_Q \mathcal{A} = \mathcal{A}^k, \qquad
\mathcal{X} *_Q \mathcal{A} *_Q \mathcal{X} = \mathcal{X}, \qquad
\mathcal{A} *_Q \mathcal{X} = \mathcal{X} *_Q \mathcal{A}.
\]
In this case, the QT-Drazin inverse of $\mathcal{A}$ is denoted by $\mathcal{A}^D$.
\end{definition}

Furthermore, we explore the relationship between the QT-Drazin inverse of the $z$-block circulant matrix $ \mathtt{bcirc_z}(\mathcal A)$ and the $z$-block circulant representation of $\mathcal A^D$. Before proceeding, we first present the following lemma.
    

\begin{lemma}\label{AKA-1}
     Let $\mathcal A\in\mathbb{Q}^{n\times n\times n_3}$. Then $\mathtt{bcirc_z}(\mathcal A^k)=\mathtt{bcirc_z}(\mathcal A)^k$. If $\mathcal A$ is invertible, we obtain $\mathtt{bcirc_z}(\mathcal A^{-1})=\mathtt{bcirc_z}(\mathcal A)^{-1}.$
     \begin{proof}
         Note that for any quaternion tensor $\mathcal A\in\mathbb Q^{n\times n\times n_3}$, we utilize Lemma \ref{newrelation} to derive
\begin{align*}
   \mathtt{bcirc_z}(\mathcal A^2)=\mathtt{bcirc_z}(\mathcal A*_Q\mathcal A)=\mathtt{bcirc_z}(\mathcal A)\mathtt{bcirc_z}(\mathcal A) =\mathtt{bcirc_z}(\mathcal A)^2.
\end{align*}
By induction, it follows that $\mathtt{bcirc_z}(\mathcal A^k)=\mathtt{bcirc_z}(\mathcal A)^k$, for any positive integer $k\geq0$. Furthermore, since $\mathcal A*\mathcal A^{-1}=\mathcal I$, which yields the following relation by Lemma \ref{newrelation} and Lemma \ref{exchangeinverse}:
\begin{align*}
\mathtt{bcirc_z}(\mathcal A)\mathtt{bcirc_z}(\mathcal A^{-1})=\mathtt{bcirc_z}(\mathcal I)=I_{nn_3}=\mathtt{bcirc_z}(\mathcal A)\mathtt{bcirc_z}(\mathcal A)^{-1}.
\end{align*}
Therefore, the desired result follows.
     \end{proof}
\end{lemma}

In what follows, we investigate the relation between $ \mathtt{bcirc_z}(\mathcal A)^D$ and $ \mathtt{bcirc_z}(\mathcal A^D)$, as follows.
\begin{lemma}\label{zad}
    Let $\mathcal A\in\mathbb{Q}^{n\times n\times n_3}$ and $\mathcal A^D\in\mathbb{Q}^{n\times n\times n_3}$ with $Ind_{QT}(\mathcal A)=k$. Then
    \begin{align*}
        \mathtt{bcirc_z}(\mathcal A)^D=\mathtt{bcirc_z}(\mathcal A^D).
    \end{align*}
\begin{proof}
      According to the QT-Drazin inverse, the following three equations can be obtained:
\begin{align*}
    &\mathtt{bcirc_z}(\mathcal A)^D=\mathtt{bcirc_z}(\mathcal A)^D\mathtt{bcirc_z}(\mathcal A)\mathtt{bcirc_z}(\mathcal A)^D,\\
    &\mathtt{bcirc_z}(\mathcal A)\mathtt{bcirc_z}(\mathcal A)^D=\mathtt{bcirc_z}(\mathcal A)^D\mathtt{bcirc_z}(\mathcal A),\\
    &\mathtt{bcirc_z}(\mathcal A)^k\mathtt{bcirc_z}(\mathcal A)^D\mathtt{bcirc_z}(\mathcal A)=\mathtt{bcirc_z}(\mathcal A)^k.
\end{align*}
On the other hand, by utilizing Lemma \ref{newrelation}, we also have
\begin{align*}
    &\mathtt{bcirc_z}(\mathcal A^D)=\mathtt{bcirc_z}(\mathcal A^D)\mathtt{bcirc_z}(\mathcal A)\mathtt{bcirc_z}(\mathcal A^D),\\
    &\mathtt{bcirc_z}(\mathcal A)\mathtt{bcirc_z}(\mathcal A^D)=\mathtt{bcirc_z}(\mathcal A^D)\mathtt{bcirc_z}(\mathcal A),\\
    &\mathtt{bcirc_z}(\mathcal A^k)\mathtt{bcirc_z}(\mathcal A^D)\mathtt{bcirc_z}(\mathcal A)=\mathtt{bcirc_z}(\mathcal A^k).
\end{align*}
By Lemma \ref{AKA-1}, we obtain $$\mathtt{bcirc_z}(\mathcal A^k)\mathtt{bcirc_z}(\mathcal A^D)\mathtt{bcirc_z}(\mathcal A)=\mathtt{bcirc_z}(\mathcal A^k)$$ equals to $$\mathtt{bcirc_z}(\mathcal A)^k\mathtt{bcirc_z}(\mathcal A)^D\mathtt{bcirc_z}(\mathcal A)=\mathtt{bcirc_z}(\mathcal A)^k.$$
Due to the uniqueness of the Drazin inverse, the desired result is obtained.  
\end{proof}
\end{lemma}

The following main theorem establishes a connection between the QT-index of a tensor and the indices of the diagonal blocks in its associated block-diagonalized matrix, thereby enabling a more transparent analysis of the QT-index from the matrix viewpoint.
\begin{theorem}\label{indqtmax}
Let $\mathcal{A}\in\mathbb{Q}^{n\times n\times n_3}$ and 
$\mathtt{bcirc_z}(\mathcal{A})$ be its $z$-block circulant matrix. 
Under the discrete Fourier transform along the third mode, we obtain
\[
(F_{n_3}\otimes I_n)\,
\mathtt{bcirc_z}(\mathcal{A})\,
(F_{n_3}^{*}\otimes I_n)
=diag
\big(
\hat{A}^{(1)},
\dots,
\hat{A}^{(n_3)}
\big),
\]
where $\hat{A}^{(i)}\in\mathbb{Q}^{n\times n}$ are the frontal slices of $\hat{\mathcal{A}}$.
Then the QT-index of \(\mathcal{A}\) satisfies
\[
\operatorname{Ind}_{QT}(\mathcal{A})=\max_{1\le i\le n_3}\operatorname{Ind}\big(\hat{A}^{(i)}\big),
\]
where \(\operatorname{Ind}(\cdot)\) denotes the index of a square matrix, i.e., the smallest nonnegative integer \(k\) such that \(\operatorname{rank}(M^{k+1})=\operatorname{rank}(M^{k})\).
\end{theorem}
\begin{proof}
From Lemma \ref{newrelation} and Lemma \ref{AKA-1}, we have \(\mathtt{bcirc_z}(\mathcal{A}^k)=\mathtt{bcirc_z}(\mathcal{A})^k\) for any positive integer \(k\). By Lemma \ref{zdc}, the unitary matrix $F_{n_3}\otimes I_n$ diagonalizes 
$\mathtt{bcirc_z}(\mathcal{A})$, namely,
\[
(F_{n_3}\otimes I_n)\,\mathtt{bcirc_z}(\mathcal{A})\,(F_{n_3}^*\otimes I_n)=diag\big(\hat{A}^{(1)},\dots,\hat{A}^{(n_3)}\big)\equiv\hat{D}.
\]
Furthermore,
\[
(F_{n_3}\otimes I_n)\,\mathtt{bcirc_z}(\mathcal{A}^k)\,(F_{n_3}^*\otimes I_n)=\hat{D}^k=diag\big((\hat{A}^{(1)})^k,\dots,(\hat{A}^{(n_3)})^k\big).
\]
Since rank is invariant under unitary similarity, it follows that
\[
\operatorname{rank}\big(\mathtt{bcirc_z}(\mathcal{A}^k)\big)=\operatorname{rank}(\hat{D}^k)=\sum_{i=1}^{n_3}\operatorname{rank}\big((\hat{A}^{(i)})^k\big).
\]

On the other hand, by Definition \ref{rankqt}, the QT-rank of $\mathcal{A}^k$
coincides with the rank of $\mathtt{bcirc_z}(\mathcal{A}^k)$. 
Since $\hat{D}^k$ is block diagonal, we have
\[
\operatorname{rank}_{QT}(\mathcal{A}^k)=\sum_{i=1}^{n_3}\operatorname{rank}\big((\hat{A}^{(i)})^k\big)=\operatorname{rank}\big(\mathtt{bcirc_z}(\mathcal{A}^k)\big).
\]
Consequently, \(\operatorname{rank}_{QT}(\mathcal{A}^{k+1})=\operatorname{rank}_{QT}(\mathcal{A}^k)\) if  \[
\operatorname{rank}\big((\hat{A}^{(i)})^{k+1}\big)
=
\operatorname{rank}\big((\hat{A}^{(i)})^k\big),
\quad 1\le i\le n_3,
\]
that is, $k\ge \operatorname{Ind}(\hat{A}^{(i)})$ for all $i$. 
Hence, the minimal such $k$ equals $\max\limits_{1\le i\le n_3}\operatorname{Ind}\big(\hat{A}^{(i)}\big),$
which, by Definition \ref{QT_index}, coincides with 
$\operatorname{Ind}_{QT}(\mathcal{A})$.
\end{proof}

The following main result provides a formula expressing the Drazin inverse of a quaternion tensor in terms of its Moore–Penrose inverse, thereby making it possible to compute the Drazin inverse of a quaternion tensor directly in MATLAB.
\begin{theorem}\label{MPMATLAB}
Let \(\mathcal{A}\in\mathbb{Q}^{n\times n\times n_3}\) be a third-order quaternion tensor with \(\operatorname{Ind}_{QT}(\mathcal{A})=k\). Then, for any integer \(l\ge k\),
\[
\mathcal{A}^D = \mathcal{A}^l *_Q (\mathcal{A}^{2l+1})^\dagger *_Q \mathcal{A}^l,
\]
where \(\mathcal{A}^D\) denotes the QT-Drazin inverse of $\mathcal{A}$, 
and $(\cdot)^\dagger$ denotes the QT-Moore--Penrose inverse.
\end{theorem}
\begin{proof}
Let \(B = \mathtt{bcirc_z}(\mathcal{A})\in\mathbb{Q}^{nn_3\times nn_3}\). By Lemma \ref{AKA-1} and Lemma \ref{zad}, we have
\[
\mathtt{bcirc_z}(\mathcal{A}^D)=B^D,\qquad 
\mathtt{bcirc_z}(\mathcal{A}^\dagger)=B^\dagger,\qquad 
\mathtt{bcirc_z}(\mathcal{A}^l)=B^l.
\]
Hence, the desired equality \(\mathcal{A}^D=\mathcal{A}^l *_Q (\mathcal{A}^{2l+1})^\dagger *_Q \mathcal{A}^l\) holds if and only if
\begin{align}\label{BDBL}
B^D = B^l (B^{2l+1})^\dagger B^l.
\end{align}
Since \(B\) is a $z$-block circulant matrix, by Lemma \ref{zdc}, there exists a unitary matrix \(F = F_{n_3}\otimes I_n\) such that
\[
F B F^* = diag(\hat{A}^{(1)},\hat{A}^{(2)},\dots,\hat{A}^{(n_3)}),
\]
where each \(\hat{A}^{(i)}\in\mathbb{Q}^{n\times n}\) is the \(i\)-th frontal slice of \(\hat{\mathcal{A}}\). Since $F$ is unitary, matrix powers, the Moore--Penrose inverse, and the Drazin inverse are all invariant under unitary similarity transformations. Consequently, applying these invariance properties to the unitary block diagonalization of $B$, we obtain
\begin{align}\label{FDF*}
&F B^l F^* = diag\big((\hat{A}^{(1)})^l,\dots,(\hat{A}^{(n_3)})^l\big),\notag\\
&F (B^{2l+1})^\dagger F^* = diag\Big(\big((\hat{A}^{(1)})^{2l+1}\big)^\dagger,\dots,\big((\hat{A}^{(n_3)})^{2l+1}\big)^\dagger\big),\notag\\ 
&F B^D F^* = diag\big((\hat{A}^{(1)})^D,\dots,(\hat{A}^{(n_3)})^D\big).
\end{align}

By Theorem \ref{indqtmax}, \(\operatorname{Ind}_{QT}(\mathcal{A})=\max\limits_{1\le i\le n_3}
\Ind(\hat{A}^{(i)})=k\), it follows that $\Ind(B)=k$, and \(\operatorname{Ind}(\hat{A}^{(i)})\le k\) for each \(i\). By \cite[Proposition 3.1]{K2014}, for any 
$l\ge \max\limits_{1\le i\le n_3}\operatorname{Ind}(\hat{A}^{(i)})$, one has
\begin{align*}
(\hat{A}^{(i)})^D = (\hat{A}^{(i)})^l \big((\hat{A}^{(i)})^{2l+1}\big)^\dagger (\hat{A}^{(i)})^l.
\end{align*}
Applying the above result to each diagonal block of \eqref{FDF*} and assembling the resulting identities yields
\begin{align*}
F B^D F^* &= \operatorname{diag}\big((\hat{A}^{(1)})^l \big((\hat{A}^{(1)})^{2l+1}\big)^\dagger (\hat{A}^{(1)})^l,\dots,(\hat{A}^{(n_3)})^l \big((\hat{A}^{(n_3)})^{2l+1}\big)^\dagger (\hat{A}^{(n_3)})^l\big)\\&
= F\big(B^l(B^{2l+1})^\dagger B^l\big)F^*,
\end{align*}where $l\ge k$. Left multiplication by $F^*$ and right multiplication by $F$ yield \eqref{BDBL}. 
The desired identity follows by applying $\mathtt{bcirc_z}^{-1}$.
\end{proof}

Moreover, note that any $n\times n$ quaternion matrix admits a Jordan decomposition~\cite{Jordan form}. 
Motivated by this fact, we utilize the properties of $z$-block circulant matrices to extend this concept to quaternion tensors via the QT-product. 

\begin{lemma}[\textbf{QT-Jordan decomposition}]\label{pjp}
Let $\mathcal{A}\in\mathbb{Q}^{n\times n\times n_3}$. Then there exists an invertible quaternion tensor $\mathcal{P}\in\mathbb{Q}^{n\times n\times n_3}$ such that  
\[
\mathcal{A}=\mathcal{P} *_Q \mathcal{J} *_Q \mathcal{P}^{-1},
\]
where $\mathcal{J}$ denotes the QT-Jordan canonical form of $\mathcal{A}$.
\begin{proof}
    Let $A$ be an $n\times n$ quaternion matrix with distinct (right) eigenvalues, whose imaginary parts are nonnegative. Then there exists an invertible quaternion matrix $P$ such that $A=PJP^{-1}.$
    By utilizing Lemma \ref{newrelation} and Lemma \ref{exchangeinverse}, we obtain
    \begin{align*}
        \mathtt{bcirc_z}(\mathcal A)&=\mathtt{bcirc_z}(\mathcal P)\mathtt{bcirc_z}(\mathcal J)\mathtt{bcirc_z}(\mathcal P)^{-1}\\&=\mathtt{bcirc_z}(\mathcal P)\mathtt{bcirc_z}(\mathcal J)\mathtt{bcirc_z}(\mathcal P^{-1}),
    \end{align*} where $\mathtt{bcirc_z}(\mathcal P)$ is invertible and $\mathtt{bcirc_z}(\mathcal J)$ is an upper-triangular quaternion matrix and the main diagonal entries are $\lambda_1^{(i)},~\lambda_2^{(i)},\cdots, \lambda_t^{(i)}$ associated with $\mathtt{bcirc_z}(\mathcal A)$ and $t\leq nn_3$. Let $\mathcal{T}, \mathcal{J}$ be quaternion tensors corresponding to 
        $\mathtt{bcirc_z}(\mathcal{T})$ and $\mathtt{bcirc_z}(\mathcal{J}),$ respectively. Then, by Lemma \ref{newrelation} and Lemma \ref{exchangeinverse}, we obtain
    \begin{align*}
        \mathcal A=\mathcal P*_Q\mathcal J*_Q\mathcal P^{-1},
    \end{align*} where $\mathcal P$ is invertible and $\mathcal J$ is f-upper-triangular quaternion tensor. This completes the proof.
\end{proof}
\end{lemma}

Now, using the preceding results, we present the following decomposition of the QT-Drazin inverse of the quaternion tensor $\mathcal A$.
\begin{theorem}\label{Thm3.1}
Let $\mathcal A\in\mathbb{Q}^{n\times n\times n_3}$ with $Ind_{QT}(\mathcal A)=k$. Then there exists an invertible quaternion tensor $\mathcal{P}\in\mathbb{Q}^{n\times n\times n_3}$ such that the QT-Drazin inverse of $\mathcal A$ has the following form
$$\mathcal A^D=\mathcal P*_Q\mathcal J^D*_Q\mathcal P^{-1}.$$
\begin{proof}
    By Lemma \ref{newrelation} and Lemma \ref{pjp}, we have
    \begin{align}\label{yue11}
    \mathtt{bcirc_z}(\mathcal A)=\mathtt{bcirc_z}(\mathcal P)\mathtt{bcirc_z}(\mathcal J)\mathtt{bcirc_z}(\mathcal P^{-1}).
    \end{align}
Accordingly, by Lemma \ref{AKA-1}, we derive
\begin{align*}
    \mathtt{bcirc_z}(\mathcal A)^D=\mathtt{bcirc_z}(\mathcal P)\mathtt{bcirc_z}(\mathcal J)^D\mathtt{bcirc_z}(\mathcal P)^{-1}.
\end{align*}
Moreover, by Lemma \ref{AKA-1} and Lemma \ref{zad}, it follows that
\begin{align*}
    \mathtt{bcirc_z}(\mathcal A^D)=\mathtt{bcirc_z}(\mathcal P)\mathtt{bcirc_z}(\mathcal J^D)\mathtt{bcirc_z}(\mathcal P^{-1}).
\end{align*}
It can be directly verified that $\mathcal A^D=\mathcal P*_Q\mathcal J^D*_Q\mathcal{P}^{-1}$ by Lemma \ref{newrelation}.
\end{proof}
\end{theorem}

The following corollary is a direct consequence of Theorem \ref{Thm3.1}, and has also been discussed in \cite{PJOOma2024}.
\begin{corollary}
    Let $\mathcal A\in\mathbb{Q}^{n\times n\times n_3}$ with $\Ind_{QT}(\mathcal A)=1$. Then the QT-group inverse of $\mathcal A$ has the following form
    \begin{align*}
        \mathcal A^{\#}=\mathcal P*_Q\mathcal J^{\#}*_Q\mathcal P^{-1}.
    \end{align*}
\end{corollary}

\section{Perturbation analysis of the QT-Drazin inverse}

In this section, the perturbation analysis of the QT-Drazin inverse of quaternion tensors under the QT-product is conducted, and the corresponding perturbation bounds are obtained. Before proceeding to the main results, a proposition and several lemmas are introduced.
\begin{proposition}\label{yz1}
Let $A \in \mathbb{Q}^{n \times n}$
and let $\lambda$ is a right eigenvalue of $A$. If $\rho(A)<1$, then $I+A$ is invertible.
\end{proposition}
\begin{proof}
Suppose there exists a nonzero vector $x$ such that 
\[
Ax= x\lambda.
\]
Then,
\[
(I + A)x = Ix + Ax = x + x\lambda  = x(1+\lambda).
\]
Since $\rho(A)<1$, we have $1+\lambda\neq0$, i.e., $det(I+A)\neq0$, i.e., $I+A$ is invertible.
\end{proof}

\begin{lemma}\label{yz2}
Let $\mathcal I\in\mathbb{Q}^{n\times n\times n_3}$ be the identity quaternion tensor. Then
\begin{align*}
    \mathcal I+\mathcal A ~\text{is invertible} \Longleftrightarrow I_{nnn_3}+\mathtt{bcirc_z}(\mathcal A) ~\text{is invertible}
    \end{align*}
for any $\mathcal A\in\mathbb{Q}^{n\times n\times n_3}$.
    \begin{proof}
    By Lemma \ref{exchangeinverse}, we have $\mathcal I+\mathcal A~\text{is invertible} \Longleftrightarrow \mathtt{bcirc_z}(\mathcal I+\mathcal A)~\text{is invertible}.$ Since $$\mathtt{bcirc_z}(\mathcal I+\mathcal A)=\mathtt{bcirc_z}(\mathcal I)+\mathtt{bcirc_z}(\mathcal A)=I_{nnn_3}+\mathtt{bcirc_z}(\mathcal A),$$ then the desired result is obtained.
    \end{proof}
\end{lemma}

\begin{lemma}\label{ABBARHO}
For any $\mathcal A, \mathcal B\in\mathbb{Q}^{n\times n\times n_3}$,
\begin{align*}
    0<\scalebox{1.3}{$\rho$}_{QT}(\mathcal A*_Q \mathcal{B})<1 \Longleftrightarrow 0<\scalebox{1.3}{$\rho$}_{QT}(\mathcal{B}*_Q \mathcal{A})<1.
\end{align*}
\begin{proof}
    Assume that $\lambda\neq0$ is the right eigenvalue of $\mathtt{bcirc_z}(\mathcal A)\mathtt{bcirc_z}(\mathcal{B})$, then there exists eigenvector $x\neq0$ such that 
    \begin{align}\label{eigen}
    \Big(\mathtt{bcirc_z}(\mathcal A)\mathtt{bcirc_z}(\mathcal{B})\Big)x=x\lambda.
    \end{align}
    Notice that $\mathtt{bcirc_z}(\mathcal B)x\neq0$, left-multiplying equation \eqref{eigen} by $\mathtt{bcirc_z}(\mathcal B)$ yields
    \begin{align*}
    \mathtt{bcirc_z}(\mathcal B)\mathtt{bcirc_z}(\mathcal A)\Big(\mathtt{bcirc_z}(\mathcal B)x\Big)&=\mathtt{bcirc_z}(\mathcal B)\Big(\mathtt{bcirc_z}(\mathcal A)\mathtt{bcirc_z}(\mathcal B)x\Big)\\&=\mathtt{bcirc_z}(\mathcal B)(x\lambda)\\&=(\mathtt{bcirc_z}(\mathcal B)x)\lambda.
    \end{align*}
    Thus, $\mathtt{bcirc_z}(\mathcal B)x$ is the eigenvector of $\mathtt{bcirc_z}(\mathcal B)\mathtt{bcirc_z}(\mathcal A),$ and $\mathtt{bcirc_z}(\mathcal A)\mathtt{bcirc_z}(\mathcal B)$ has the same nonzero eigenvalues as $\mathtt{bcirc_z}(\mathcal B)\mathtt{bcirc_z}(\mathcal A).$
    In addition, utilizing Definition \ref{defrhoqt} and Lemma \ref{newrelation}, we obtain
    \begin{align*}
    0<\scalebox{1.3}{$\rho$}_{QT}(\mathcal A*_Q \mathcal{B})<1\Longleftrightarrow 0<\scalebox{1.3}{$\rho$}\Big(\mathtt{bcirc_z}(\mathcal A)\mathtt{bcirc_z}(\mathcal{B})\Big)<1
    \end{align*} and 
    \begin{align*}
    0<\scalebox{1.3}{$\rho$}_{QT}(\mathcal B*_Q \mathcal{A})<1\Longleftrightarrow 0<\scalebox{1.3}{$\rho$}\Big(\mathtt{bcirc_z}(\mathcal B)\mathtt{bcirc_z}(\mathcal{A})\Big)<1.
    \end{align*}
    Therefore, we finish the proof.
\end{proof}
\end{lemma}

Subsequently, we provide the following theorem on the perturbation theory of quaternion tensors.
\begin{theorem}\label{Thm3.2}
    Let $\mathcal A, \mathcal B, \mathcal E\in\mathbb{Q}^{n\times n\times n_3}$, and let $\mathcal A^D$ denote the QT-Drazin inverse of $\mathcal A$ with $Ind_{QT}(\mathcal A)=k$. Suppose that $\mathcal E=\mathcal A*_Q\mathcal A^D*_Q\mathcal E*_Q\mathcal A*_Q\mathcal A^D$, and define $\mathcal B=\mathcal A+\mathcal E$. If the QT spectral radius satisfies $0<\scalebox{1.3}{$\rho$}_{QT}(\mathcal{A}^D*_Q \mathcal{E})<1$, then
\begin{align*}
    \mathcal A*_Q\mathcal A^D&=\mathcal B*_Q\mathcal B^D,\\
    \mathcal B^D-\mathcal A^D&=-\mathcal B^D*_Q\mathcal E*_Q\mathcal A^D=-\mathcal A^D*_Q\mathcal E*_Q\mathcal B^D,\\
    \mathcal B^D&=(\mathcal I+\mathcal A^D*_Q\mathcal E)^{-1}*_Q\mathcal A^D=\mathcal A^D*_Q(\mathcal I+\mathcal E*_Q\mathcal A^D)^{-1}.
\end{align*}
\begin{proof}
    By assumption, we have 
    \[
    \begin{cases}
        \mathcal B=\mathcal A+\mathcal E=\mathcal A+\mathcal A*_Q\mathcal A^D*_Q\mathcal E=\mathcal A*_Q(\mathcal I+\mathcal A^D*_Q\mathcal E),\\
        \mathcal B=\mathcal A+\mathcal E=\mathcal A+\mathcal E*_Q\mathcal A*_Q\mathcal A^D=(\mathcal I+\mathcal E*_Q\mathcal A^D)*_Q\mathcal A.
    \end{cases}
    \]
As $0<\scalebox{1.3}{$\rho$}_{QT}(\mathcal{A}^D *_Q \mathcal{E}) < 1$, we combine Proposition \ref{yz1} and Lemma \ref{yz2} to obtain that $\mathcal I + \mathcal A^D *_Q \mathcal E$ is nonsingular. Similarly, it can be verified that $\mathcal I + \mathcal E *_Q \mathcal A^D$ is also nonsingular by Lemma \ref{ABBARHO}.

By Lemma~\ref{pjp} and Theorem~\ref{Thm3.1}, $\mathcal A$ and $\mathcal A^D$ can be expressed as
\[
\mathcal A = \mathcal P *_Q \mathcal J *_Q \mathcal P^{-1}, 
\qquad 
\mathcal A^D = \mathcal P *_Q \mathcal J^D *_Q \mathcal P^{-1},
\]
where $\mathcal J$ denotes the F-upper-triangular quaternion tensor. 
The diagonal entries of $\hat{J}_{(i)}$ are given by the set 
$\{\lambda_{(i)}^1, \lambda_{(i)}^2, \ldots, \lambda_{(i)}^s\}$ ($s \leq n$), 
which correspond to the distinct right eigenvalues of $\hat{A}_{(i)}$.
Furthermore, we utilize Lemma \ref{zdc} to obtain
\begin{align}\label{FJF}
    &(F_{n_3}\otimes I_n)\mathtt{bcirc_z}(\mathcal J)(F_{n_3}^*\otimes I_n)=\begin{bmatrix}
    \hat{J}_{(1)}&&&\\&\hat{J}_{(2)}&&\\ &&\ddots&\\&&&\hat{J}_{(n_3)}
\end{bmatrix},
\end{align}
where the $i$-th entry of $diag(\hat{\mathcal J})$  $:=\hat{J}_{(i)}=\begin{bmatrix}
        \hat{C}_{(i)}&O\\O&O
\end{bmatrix}$ for $i=1,2,\cdots,n_3$.

In addition, applying Lemma \ref{zad} and Lemma \ref{zdc}, we have

\begin{align}\label{FJDF}
    &(F_{n_3}\otimes I_n)\mathtt{bcirc_z}(\mathcal J^D)(F_{n_3}^*\otimes I_n)\\\notag&=(F_{n_3}\otimes I_n)\mathtt{bcirc_z}(\mathcal J)^D(F_{n_3}^*\otimes I_n)=\begin{bmatrix}
\hat{J}_{(1)}^D&&&\\&\hat{J}_{(2)}^D&&\\ &&\ddots&\\&&&\hat{J}_{(n_3)}^D
\end{bmatrix},
\end{align}
where the $i$-th entry of $diag(\hat{\mathcal J}^D)$$:=\hat{J}_{(i)}^D=\begin{bmatrix}
        \hat{C}_{(i)}^{-1}&O\\O&O
    \end{bmatrix}$ for $i=1,2,\cdots,n_3$.
    
Notice that $\mathcal A*_Q\mathcal A^D=\mathcal P*_Q\mathcal J*_Q\mathcal J^D*_Q\mathcal P^{-1}$. Meanwhile, let $\mathcal E=\mathcal P*_Q\mathcal G*_Q\mathcal P^{-1}$. By Lemma \ref{newrelation}, we derive
\[
\begin{cases}
    \mathtt{bcirc_z}(\mathcal A*_Q\mathcal{A^D})=\mathtt{bcirc_z}(\mathcal A)\mathtt{bcirc_z}(\mathcal A^D)=\mathtt{bcirc_z}(\mathcal P)\mathtt{bcirc_z}(\mathcal J)\mathtt{bcirc_z}(\mathcal J^D)\mathtt{bcirc_z}(\mathcal P^{-1}),\\
    \mathtt{bcirc_z}(\mathcal E)=\mathtt{bcirc_z}(\mathcal P)\mathtt{bcirc_z}(\mathcal G)\mathtt{bcirc_z}(\mathcal P^{-1}).
\end{cases}
\]
Since
$\mathcal E=\mathcal A*_Q\mathcal A^D*_Q\mathcal E*_Q\mathcal A*_Q\mathcal A^D$, 
we obtain
$$\mathtt{bcirc_z}(\mathcal E)=\mathtt{bcirc_z}(\mathcal P)\mathtt{bcirc_z}(\mathcal J)\mathtt{bcirc_z}(\mathcal J^D)\mathtt{bcirc_z}(\mathcal G)\mathtt{bcirc_z}(\mathcal J)\mathtt{bcirc_z}(\mathcal J^D)\mathtt{bcirc_z}(\mathcal P^{-1}).$$
Therefore, by Lemma \ref{exchangeinverse} and Lemma \ref{AKA-1}, we conclude that
$$\mathtt{bcirc_z}(\mathcal G)=\mathtt{bcirc_z}(\mathcal J)\mathtt{bcirc_z}(\mathcal J^D)\mathtt{bcirc_z}(\mathcal G)\mathtt{bcirc_z}(\mathcal J)\mathtt{bcirc_z}(\mathcal J^D),$$ 
and 
\begin{align}\label{zdiagonal}
&(F_{n_3}\otimes I_n)\mathtt{bcirc_z}(\mathcal G)(F_{n_3}^*\otimes I_n)\\\notag&=(F_{n_3}\otimes I_n)\mathtt{bcirc_z}(\mathcal J)\mathtt{bcirc_z}(\mathcal J^D)\mathtt{bcirc_z}(\mathcal G)\mathtt{bcirc_z}(\mathcal J)\mathtt{bcirc_z}(\mathcal J^D)(F_{n_3}^*\otimes I_n).
\end{align}
By Lemma \ref{zdc}, one observes that
\begin{align}\label{FEF}
(F_{n_3}\otimes I_n)\mathtt{bcirc_z}(\mathcal G)(F_{n_3}^*\otimes I_n)=\begin{bmatrix}
    \hat{G}_{(1)}&&&\\&\hat{G}_{(2)}&&\\ &&\ddots&\\&&&\hat{G}_{(n_3)}
\end{bmatrix},
\end{align}
where the $i$-th entry of $diag(\hat{\mathcal G})$$:=\hat{G}_{(i)}=\begin{bmatrix}
    G_{(i)}^{11}&G_{(i)}^{12}\\G_{(i)}^{21}&G_{(i)}^{22}
\end{bmatrix}, i=1,2,\cdots,n_3.$

Moreover, substituting \eqref{FJF}, \eqref{FJDF}, and \eqref{FEF} into \eqref{zdiagonal}, we obtain $\hat{G}_{(i)}=\hat{J}_{(i)}\hat{J}^D_{(i)}\hat{G}_{(i)}\hat{J}_{(i)}\hat{J}^D_{(i)}$. Hence, it directly follows that
\begin{align*}
   \hat{G}_{(i)}= \begin{bmatrix}
    G_{(i)}^{11}&G_{(i)}^{12}\\G_{(i)}^{21}&G_{(i)}^{22}
    \end{bmatrix}=\begin{bmatrix}
        \hat{C}_{(i)}&O\\O&O
\end{bmatrix}\begin{bmatrix}
        \hat{C}_{(i)}^{-1}&O\\O&O
\end{bmatrix} \begin{bmatrix}
    G_{(i)}^{11}&G_{(i)}^{12}\\G_{(i)}^{21}&G_{(i)}^{22}
    \end{bmatrix}\begin{bmatrix}
        \hat{C}_{(i)}&O\\O&O
\end{bmatrix}\begin{bmatrix}
        \hat{C}_{(i)}^{-1}&O\\O&O
\end{bmatrix}=\begin{bmatrix}
    G_{(i)}^{11}&O\\O&O
    \end{bmatrix}.
\end{align*}
Since $\mathcal B=\mathcal A+\mathcal E=\mathcal P*_Q(\mathcal J+\mathcal G)*_Q\mathcal P^{-1}$, we obtain
$$\mathtt{bcirc_z}(\mathcal B)=\mathtt{bcirc_z}(\mathcal P)\mathtt{bcirc_z}(\mathcal J+\mathcal G)\mathtt{bcirc_z}(\hat{\mathcal P}^{-1}).$$ Furthermore, from \eqref{FJF} and \eqref{FEF}, we can easily derive 
\begin{align}\label{JE}
    (F_{n_3}\otimes I_n)\mathtt{bcirc_z}(\mathcal J+\mathcal G)(F_{n_3}^*\otimes I_n)=\begin{bmatrix}
    \hat{J}_{(1)}+\hat{G}_{(1)}& & &\\
    & \hat{J}_{(2)}+\hat{G}_{(2)} & & \\
    & & \ddots &\\
    & & & \hat{J}_{(n_3)}+\hat{G}_{(n_3)}
\end{bmatrix},
\end{align}
where $\hat{J}_{(i)}+\hat{G}_{(i)}=\begin{bmatrix}
    \hat{C}_{(i)}+{G}_{(i)}^{11}&O\\O&O
\end{bmatrix},$ for $i=1,2,\cdots,n_3.$

By Lemma~\ref{exchangeinverse}, the invertibility of $\mathcal I + \mathcal A^D *_Q \mathcal E$ ensures that $\mathtt{bcirc_z}(\mathcal I + \mathcal A^D *_Q \mathcal E)$ is invertible. It is then straightforward to see that
\begin{align*}
    \mathtt{bcirc_z}(\mathcal I+\mathcal A^D*_Q\mathcal E)
    &= \mathtt{bcirc_z}(\mathcal I)+ \mathtt{bcirc_z}({\mathcal A}^D)\mathtt{bcirc_z}(\mathcal E)
    \\&=\mathtt{bcirc_z}(\mathcal P)\mathtt{bcirc_z}(\mathcal I)\mathtt{bcirc_z}(\mathcal P^{-1})
    \\&+\mathtt{bcirc_z}(\mathcal P)\mathtt{bcirc_z}(\mathcal J^D)\mathtt{bcirc_z}(\mathcal P^{-1})\mathtt{bcirc_z}(\mathcal P)\mathtt{bcirc_z}(\mathcal G)\mathtt{bcirc_z}(\mathcal P^{-1})
    \\&=\mathtt{bcirc_z}(\mathcal P)\Big(\mathtt{bcirc_z}(\mathcal I)+\mathtt{bcirc_z}(\mathcal J^D)\mathtt{bcirc_z}(\mathcal G\Big)\mathtt{bcirc_z}(\mathcal P^{-1}),
\end{align*}
and thus $\mathtt{bcirc_z}(\mathcal I)+\mathtt{bcirc_z}(\mathcal J^D)\mathtt{bcirc_z}(\mathcal G)$ is invertible and we expand it as follows:
\begin{align}\label{zhao1}
(F_{n_3}^*\otimes I_n)\begin{bmatrix}
\begin{pmatrix}
    \hat{I}_{r_1}+(\hat{C}_{(1)})^{-1}G_{(1)}^{11}&O\\O&\hat{I}_{n-r_1}
    \end{pmatrix}& & \\
    &\ddots &\\
    & &\begin{pmatrix}
    \hat{I}_{r_{n_3}}+(\hat{C}_{(n_3)})^{-1}G_{(n_3)}^{11}&O\\O&\hat{I}_{n-r_{n_3}}
    \end{pmatrix}
\end{bmatrix}(F_{n_3}\otimes I_n),
\end{align}
where $r_{i}$ is the number of non-zero eigenvalues of $\hat{A}_{(i)}$, $i=1,2,\cdots,n_3.$ Observe that $\hat{I}_{r_i}+(\hat{C}_{(i)})^{-1}G_{(i)}^{11}$ can be rewritten as $(\hat{C}_{(i)})^{-1}(\hat{C}_{(i)}+G_{(i)}^{11})$, then we can split the diagonal blocks of \eqref{zhao1} as follows:\begin{align*}
    \begin{pmatrix}
    \hat{I}_{r_{i}}+(\hat{C}_{(i)})^{-1}G_{(i)}^{11}&O\\O&\hat{I}_{n-r_{i}}
    \end{pmatrix}=\begin{pmatrix}
    (\hat{C}_{(i)})^{-1}&O\\O&\hat{I}_{n-r_{i}}
    \end{pmatrix}\begin{pmatrix}
    \hat{C}_{(i)}+G_{(i)}^{11}&O\\O&\hat{I}_{n-r_{i}}
    \end{pmatrix}.
\end{align*}
Therefore, we conclude that $\hat{C}_{(i)}+G_{(i)}^{11}$ is invertible. From \eqref{JE}, we obtain
\begin{align}\label{JED}
&(F_{n_3}\otimes I_n)\mathtt{bcirc_z}\big((\mathcal J+\mathcal G)^D\big)(F_{n_3}^*\otimes I_n)\\\notag&=\begin{bmatrix}
\begin{pmatrix}
    (\hat{C}_{(1)}+G_{(1)}^{11})^{-1}&O\\O&O
\end{pmatrix}& & \\
    &\ddots &\\
    & &\begin{pmatrix}
    (\hat{C}_{(n_3)}+G_{(n_3)}^{11})^{-1}&O\\O&O
\end{pmatrix}
\end{bmatrix}.
\end{align}
Furthermore, we obtain
\begin{align*}
    &\mathtt{bcirc_z}(\mathcal B)\mathtt{bcirc_z}(\mathcal B^D)
    \\&=\mathtt{bcirc_z}(\mathcal P)\mathtt{bcirc_z}(\mathcal J+\mathcal G)\mathtt{bcirc_z}\big((\mathcal J+\mathcal G)^D\big)\mathtt{bcirc_z}(\mathcal P^{-1})
    \\&=\mathtt{bcirc_z}(\mathcal P)\mathtt{bcirc_z}(\mathcal J)\mathtt{bcirc_z}({\mathcal J^D})\mathtt{bcirc_z}(\hat{\mathcal P}^{-1})
    \\&=\mathtt{bcirc_z}(\mathcal P)\mathtt{bcirc_z}(\mathcal J)\mathtt{bcirc_z}(\mathcal P^{-1})\mathtt{bcirc_z}(\mathcal P)\mathtt{bcirc_z}(\mathcal J^D)\mathtt{bcirc_z}(\mathcal P^{-1})
    \\&=\mathtt{bcirc_z}(\mathcal A)\mathtt{bcirc_z}(\mathcal A^D).
\end{align*}
In particular, from \eqref{JE} and \eqref{JED}, we obtain
\begin{align*}
    &(F_{n_3}\otimes I_n)\mathtt{bcirc_z}(\mathcal J+\mathcal G)\mathtt{bcirc_z}\big((\mathcal J+\mathcal G)^D\big)(F_{n_3}^*\otimes I_n)
    \\&=\begin{bmatrix}
\begin{pmatrix}
    (\hat{C}_{(1)}+G_{(1)}^{11})(\hat{C}_{(1)}+G_{(1)}^{11})^{-1}&O\\O&O
\end{pmatrix}& & \\
    &\ddots &\\
    & &\begin{pmatrix}
    (\hat{C}_{(n_3)}+G_{(n_3)}^{11})(\hat{C}_{(n_3)}+G_{(n_3)}^{11})^{-1}&O\\O&O
\end{pmatrix}
\end{bmatrix}
    \\&=\begin{bmatrix}
\begin{pmatrix}
    I&O\\O&O
\end{pmatrix}& & \\
    &\ddots &\\
    & &\begin{pmatrix}
    I&O\\O&O
\end{pmatrix}
\end{bmatrix}=(F_{n_3}\otimes I_n)\mathtt{bcirc_z}(\mathcal J)\mathtt{bcirc_z}(\mathcal J^D)(F_{n_3}^*\otimes I_n),
\end{align*} and hence the second equality holds.

Therefore, by Lemma \ref{newrelation}, we conclude that $\mathcal A*_Q\mathcal A^D=\mathcal B*_Q\mathcal B^D$.

Moreover, it immediately implies that
\begin{align*}
    \mathcal B^D-\mathcal A^D=\mathcal B^D*_Q\mathcal A*_Q\mathcal A^D-\mathcal B^D*_Q\mathcal B*_Q\mathcal A^D
    =-\mathcal B^D*_Q\mathcal E*_Q\mathcal A^D.
\end{align*}
Analogously, we obtain
\begin{align*}
    \mathcal B^D-\mathcal A^D=\mathcal A*_Q\mathcal A^D*_Q\mathcal B^D-\mathcal A^D*_Q\mathcal B*_Q\mathcal B^D
    =-\mathcal A^D*_Q\mathcal E*_Q\mathcal B^D.
\end{align*}
Consequently, we have 
\begin{align*}
    \mathcal B^D*_Q(\mathcal I+\mathcal E*_Q\mathcal A^D)=\mathcal A^D=
    (\mathcal I+\mathcal A^D*_Q\mathcal E)*_Q\mathcal B^D.
\end{align*}
Recall that $\mathcal I+\mathcal A^D*_Q\mathcal E$ and $\mathcal I+\mathcal E*_Q\mathcal A^D$ are nonsingular, we obtain
$$\mathcal B^D=(\mathcal I+\mathcal A^D*_Q\mathcal E)^{-1}*_Q\mathcal A^D=\mathcal A^D*_Q(\mathcal I+\mathcal E*_Q\mathcal A^D)^{-1}.$$
Therefore, the proof is finished.
\end{proof}
\end{theorem}

In addition, by Lemma \ref{rhos}, $\|\mathcal A^D*_Q\mathcal E\|_s<1$ implies $\scalebox{1.3}{$\rho$}_{QT}(\mathcal{A}^D*_Q \mathcal{E})<1$. 
Hence, the next corollary follows immediately from Theorem \ref{Thm3.2}.
\begin{corollary}\label{cor4.5!}
    Let $\mathcal A, \mathcal B, \mathcal E\in\mathbb{Q}^{n\times n\times n_3}$, and let $\mathcal A^D$ denote the QT-Drazin inverse of $\mathcal A$ with $Ind_{QT}(\mathcal A)=k$. Suppose that $\mathcal E=\mathcal A*_Q\mathcal A^D*_Q\mathcal E*_Q\mathcal A*_Q\mathcal A^D$, and define $\mathcal B=\mathcal A+\mathcal E$. If the QT spectral radius satisfies $\|\mathcal A^D*_Q\mathcal E\|_s<1$, then
\begin{align*}
    \mathcal A*_Q\mathcal A^D&=\mathcal B*_Q\mathcal B^D,\\
    \mathcal B^D-\mathcal A^D&=-\mathcal B^D*_Q\mathcal E*_Q\mathcal A^D=-\mathcal A^D*_Q\mathcal E*_Q\mathcal B^D,\\
    \mathcal B^D&=(\mathcal I+\mathcal A^D*_Q\mathcal E)^{-1}*_Q\mathcal A^D=\mathcal A^D*_Q(\mathcal I+\mathcal E*_Q\mathcal A^D)^{-1}.
\end{align*}
\end{corollary}

\begin{corollary}\label{pertu}
    Let $\mathcal A, \mathcal B, \mathcal E\in\mathbb{Q}^{n\times n\times n_3}$, and let $\mathcal A^D$ denote the QT-Drazin inverse of $\mathcal A$ with $Ind_{QT}(\mathcal A)=k$. Suppose that $\mathcal E=\mathcal A*_Q\mathcal A^D*_Q\mathcal E*_Q\mathcal A*_Q\mathcal A^D$, and define $\mathcal B=\mathcal A+\mathcal E$. If the QT spectral radius satisfies $0<\scalebox{1.3}{$\rho$}_{QT}(\mathcal{A}^D*_Q \mathcal{E})<1$ and the following condition $$\Delta:~\|(\mathcal I+\mathcal A^D*_Q\mathcal E)^{-1}\|_s\leq\frac{1}{1-\|\mathcal A^D*_Q\mathcal E\|_s}$$ holds, then
\begin{align*}
    &\frac{\|\mathcal A^D\|_s}{1+\|\mathcal A^D*_Q\mathcal E\|_s}~(\text{or}~ \frac{\|\mathcal A^D\|_s}{1+\|\mathcal E*_Q\mathcal A^D\|_s})\leq\|\mathcal B^D\|_s\leq\frac{\|\mathcal A^D\|_s}{1-\|\mathcal A^D*_Q\mathcal E\|_s}~(\text{or}~ \frac{\|\mathcal A^D\|_s}{1-\|\mathcal E*_Q\mathcal A^D\|_s}),\\
    &\frac{\|\mathcal B^D-\mathcal A^D\|_s}{\|\mathcal A^D\|_s}\leq\frac{\|\mathcal A^D*_Q\mathcal E\|_s}{1-\|\mathcal A^D*_Q\mathcal E\|_s}~(\text{or}~\frac{\|\mathcal A^D*_Q\mathcal E\|_s}{1-\|\mathcal E*_Q\mathcal A^D\|_s})~\leq\frac{\kappa\|\mathcal E\|_s/\|\mathcal A\|_s}{1-\kappa\|\mathcal E\|_s/\|\mathcal A\|_s},
    \end{align*} where $\kappa= \|\mathcal{A}\|_s\|\mathcal{A}^D\|_s.$
    
    \begin{proof}
    From Theorem \ref{Thm3.2}, we obtain
    $$\|\mathcal B^D-\mathcal A^D\|_s=\|-\mathcal A^D*_Q\mathcal E*_Q\mathcal B^D\|_s\leq\|\mathcal A^D*_Q\mathcal E\|_s\|\mathcal B^D\|_s,$$
    or
    $$\|\mathcal B^D-\mathcal A^D\|_s=\|-\mathcal B^D*_Q\mathcal E*_Q\mathcal A^D\|_s\leq\|\mathcal B^D*_Q\mathcal E\|_s\|\mathcal A^D\|_s.$$
    Divided by $\|\mathcal A^D\|_s$ on both sides, it directly follows that
    \begin{align}\label{B-A/AUPPER}
     \frac{\|\mathcal B^D-\mathcal A^D\|_s}{\|\mathcal A^D\|_s}\leq\frac{\|\mathcal A^D*_Q\mathcal E\|_s\|\mathcal B^D\|_s}{\|\mathcal A^D\|_s}
     \end{align}
     or
     \begin{align*}
     \frac{\|\mathcal B^D-\mathcal A^D\|_s}{\|\mathcal A^D\|_s}\leq\|\mathcal B^D*_Q\mathcal E\|_s\leq\|\mathcal B^D\|_s\|\mathcal E\|_s.
     \end{align*}
     Furthermore, we observe that 
     \begin{align}\label{BDleft}
     \mathcal B^D=(\mathcal I+\mathcal A^D*_Q\mathcal E)^{-1}*_Q\mathcal A^D\Longleftrightarrow\mathcal A^D=(\mathcal I+\mathcal A^D*_Q\mathcal E)*_Q\mathcal B^D 
     \end{align}
     and 
     \begin{align*}
    \mathcal B^D=\mathcal A^D*_Q(\mathcal I+\mathcal E*_Q\mathcal A^D)^{-1}\Longleftrightarrow\mathcal A^D=\mathcal B^D*_Q(\mathcal I+\mathcal E*_Q\mathcal A^D).
    \end{align*} 
    By \eqref{BDleft} and the condition $\Delta$, we derive the following results, respectively:
    \begin{align}\label{snormbdupper}
        \|\mathcal B^D\|_s&\notag\leq\|(\mathcal I+\mathcal A^D*_Q\mathcal E)^{-1}\|_s\|\mathcal A^D\|_s
        \\&\leq\frac{\|\mathcal A^D\|_s}{1-\|\mathcal A^D*_Q\mathcal E\|_s},
    \end{align}
    and
    \begin{align*}
        \|\mathcal A^D\|_s&\leq\|\mathcal I+\mathcal A^D*_Q\mathcal E\|_s\|\mathcal B^D\|_s.
    \end{align*}
Then we have
\begin{align}\label{snormbdlower}
    \|\mathcal B^D\|_s\geq\frac{\|\mathcal A^D\|_s}{\|\mathcal I+\mathcal A^D*_Q\mathcal E\|_s}.
\end{align} 
Combining \eqref{snormbdupper} and \eqref{snormbdlower}, we get
\begin{align*}
    \frac{\|\mathcal A^D\|_s}{1+\|\mathcal A^D*_Q\mathcal E\|_s}\leq\|\mathcal B^D\|_s\leq\frac{\|\mathcal A^D\|_s}{1-\|\mathcal A^D*_Q\mathcal E\|_s}.
\end{align*}
Moreover, the following inequalities can be verified in a similar manner.
    \begin{align*}
        \frac{\|\mathcal A^D\|_s}{1+\|\mathcal E*_Q\mathcal A^D\|_s}&\leq\|\mathcal B^D\|_s\leq\frac{\|\mathcal A^D\|_s}{1-\|\mathcal E*_Q\mathcal A^D\|_s},
        \\
        \frac{\|\mathcal A^D\|_s}{1+\|\mathcal A^D*_Q\mathcal E\|_s}&\leq\|\mathcal B^D\|_s\leq\frac{\|\mathcal A^D\|_s}{1-\|\mathcal E*_Q\mathcal A^D\|_s},
        \\\frac{\|\mathcal A^D\|_s}{1+\|\mathcal E*_Q\mathcal A^D\|_s}&\leq\|\mathcal B^D\|_s\leq\frac{\|\mathcal A^D\|_s}{1-\|\mathcal A^D*_Q\mathcal E\|_s}.
    \end{align*}
Consequently, inequalities \eqref{B-A/AUPPER} and \eqref{snormbdupper} yield the following result
\begin{align*}
    \frac{\|\mathcal B^D-\mathcal A^D\|_s}{\|\mathcal A^D\|_s}\leq\frac{\|\mathcal A^D*_Q\mathcal E\|_s}{1-\|\mathcal A^D*_Q\mathcal E\|_s}.
\end{align*}
Similarly, we can get another result
\begin{align*}
    \frac{\|\mathcal B^D-\mathcal A^D\|_s}{\|\mathcal A^D\|_s}&\leq\frac{\|\mathcal A^D*_Q\mathcal E\|_s}{1-\|\mathcal E*_Q\mathcal A^D\|_s}.
\end{align*}
Furthermore, we conclude that
        \begin{align*}
             \frac{\|\mathcal B^D-\mathcal A^D\|_s}{\|\mathcal A^D\|_s}\leq \frac{\|\mathcal A^D\|_s\|\mathcal E\|_s}{1-\|\mathcal A^D\|_s\|\mathcal E\|_s}=\frac{\|\mathcal A\|_s\|\mathcal A^D\|_s\|\mathcal E\|_s}{\|\mathcal A\|_s-\|\mathcal A\|_s\|\mathcal A^D\|_s\|\mathcal E\|_s}=\frac{\kappa_{\mathcal D_{QT}}\|\mathcal E\|_s/\|\mathcal A\|_s}{1-\kappa_{\mathcal D_{QT}}\|\mathcal E\|_s/\|\mathcal A\|_s}.
        \end{align*}
        The remaining parts of the proof can be completed analogously.
\end{proof}
\end{corollary}

\section{Numerical example}
In this section, based on the main theorem in Section 4 together with the related corollaries, and in conjunction with Theorem \ref{MPMATLAB}, we present the following numerical example to illustrate the preceding results.
\begin{example}
    Let $\mathcal{A}\in\mathbb{Q}^{3\times 3\times 3}$ be defined by 
\[\mathcal A(:,:,1) = \begin{bmatrix} \ -2.6081+0.811\mathbf{i}-3.0064\mathbf{j}-0.5516\mathbf{k} & -1.2826+0.4267\mathbf{i}-1.2257\mathbf{j}-0.6562\mathbf{k} & 0 \\ 0.8741-1.1228\mathbf{i}+1.8279\mathbf{j}+2.5953\mathbf{k} & 0.9547+2.5177\mathbf{i}+0.5244\mathbf{j}+0.2239\mathbf{k} & 0 \\ 0 & 0 & 0\ \end{bmatrix};\]
     \[\mathcal A(:,:,2) = \begin{bmatrix} \ 1.9585+0.0946\mathbf{i}-0.4876\mathbf{j}-0.3057\mathbf{k} & -0.2933-1.7874\mathbf{i}-2.9651\mathbf{j}-1.4229\mathbf{k} & 0 \\ 0.1321+0.2343\mathbf{i}-0.1597\mathbf{j}-0.8448\mathbf{k} & -0.6288+1.393\mathbf{i}-1.9577\mathbf{j}+1.2823\mathbf{k} & 0 \\ 0 & 0 & 0\ \end{bmatrix};\]
     \[\mathcal A(:,:,3) = \begin{bmatrix} \ 1.3812+0.5786\mathbf{i}-0.821\mathbf{j}-1.0515\mathbf{k} & -0.8832+2.1316\mathbf{i}-0.0311\mathbf{j}+1.3483\mathbf{k} & 0 \\ 0.1775+2.7924\mathbf{i}+1.6295\mathbf{j}+1.2661\mathbf{k} & -2.0328+0.6504\mathbf{i}+2.1573\mathbf{j}-0.6349\mathbf{k} & 0 \\ 0 & 0 & 0\ \end{bmatrix}.\]
Consider a perturbation quaternion tensor $\mathcal{E}\in\mathbb{Q}^{3\times 3\times 3}$ with entries
\[\mathcal E(:,:,1) = \begin{bmatrix} \ 0.2+0.1\mathbf{i}+0.6\mathbf{j}+0.7\mathbf{k} & 0 & 0 \\ 0 & 0.8+0.3\mathbf{i}+0.5\mathbf{j}+0.8\mathbf{k} & 0 \\ 0 & 0 & 0 \end{bmatrix};\]
     \[\mathcal E(:,:,2) = \begin{bmatrix} 0 & 0 & 0 \\ 0 & 0 & 0 \\ 0 & 0 & 0 \end{bmatrix}; \quad
     \mathcal E(:,:,3) = \begin{bmatrix}  0 & 0 & 0 \\ 0 & 0 & 0 \\ 0 & 0 & 0 \end{bmatrix}.\]
We next note that $Ind_{QT}(\mathcal A)=1$. Indeed, after straightforward computations, one verifies that $\mathcal E=\mathcal A*_Q\mathcal A^D*_Q\mathcal E*_Q\mathcal A*_Q\mathcal A^D$ and $\|\mathcal A^D*_Q\mathcal E\|_s=0.4433<1$.
Hence, the assumptions of Theorem \ref{Thm3.2}, Corollary \ref{cor4.5!}, and Corollary \ref{pertu} are simultaneously satisfied. We then proceed to compute
\[\mathcal B^D(:,:,1) = \begin{bmatrix} \ -2.4081+0.911\mathbf{i}-2.4064\mathbf{j}+0.1484\mathbf{k} & -1.2826+0.4267\mathbf{i}-1.2257\mathbf{j}-0.6562\mathbf{k} & 0 \\ 0.8741-1.1228\mathbf{i}+1.8279\mathbf{j}+2.5953\mathbf{k} & 1.7547+2.8177\mathbf{i}+1.0244\mathbf{j}+1.0239\mathbf{k} & 0 \\ 0 & 0 & 0\ \end{bmatrix};\]
     \[\mathcal B^D(:,:,2) = \begin{bmatrix} \ 1.9585+0.0946\mathbf{i}-0.4876\mathbf{j}-0.3057\mathbf{k} & -0.2933-1.7874\mathbf{i}-2.9651\mathbf{j}-1.4229\mathbf{k} & 0 \\ 0.1321+0.2343\mathbf{i}-0.1597\mathbf{j}-0.8448\mathbf{k} & -0.6288+1.393\mathbf{i}-1.9577\mathbf{j}+1.2823\mathbf{k} & 0 \\ 0 & 0 & 0\ \end{bmatrix};\]
     \[\mathcal B^D(:,:,3) = \begin{bmatrix} \ 1.3812+0.5786\mathbf{i}-0.821\mathbf{j}-1.0515\mathbf{k} & -0.8832+2.1316\mathbf{i}-0.0311\mathbf{j}+1.3483\mathbf{k} & 0 \\ 0.1775+2.7924\mathbf{i}+1.6295\mathbf{j}+1.2661\mathbf{k} & -2.0328+0.6504\mathbf{i}+2.1573\mathbf{j}-0.6349\mathbf{k} & 0 \\ 0 & 0 & 0\ \end{bmatrix};\]
and
\[ \|\mathcal A^D\|_s=0.3938;\quad    \|\mathcal B^D\|_s=0.5150;\quad  \|\mathcal B^D-\mathcal A^D\|_s=0.1737.\]
Moreover, further computations yield
\[\frac{\|\mathcal A^D\|_s}{1+\|\mathcal A^D*_Q\mathcal E\|_s}=0.2728;\quad
\frac{\|\mathcal A^D\|_s}{1-\|\mathcal A^D*_Q\mathcal E\|_s}=0.7073;\]

\[\frac{\|\mathcal B^D-\mathcal A^D\|_s}{\|\mathcal A^D\|_s}=0.4412;\quad
\frac{\|\mathcal A^D*_Q\mathcal E\|_s}{1-\|\mathcal A^D*_Q\mathcal E\|_s}=0.7964;\quad
\frac{\kappa\|\mathcal E\|_s/\|\mathcal A\|_s}{1-\kappa\|\mathcal E\|_s/\|\mathcal A\|_s}=1.0047.\]
Therefore, the inequalities stated in Corollary \ref{pertu} are satisfied.
\end{example}

\section*{Funding}
\begin{itemize}
\item Daochang Zhang is supported by the National Natural Science Foundation of China (NSFC) (No. 11901079), China Postdoctoral Science Foundation (No. 2021M700751), and the Scientific and Technological Research Program Foundation of Jilin Province (No. JJKH20190690KJ; No. JJKH20220091KJ; No. JJKH20250851KJ).
\item Dijana Mosi\'c is supported by the Ministry of Science, Technological Development and Innovation, Republic of Serbia, grant number 451-03-34/2026-03/200124.
\end{itemize}

\section*{Conflict of Interest}
The authors declare that they have no potential conflict of interest.

\section*{Data Availability}
Data sharing is not applicable to this article, as no datasets were generated or analyzed during the current study.


\bibliographystyle{abbrv}

\end{document}